\documentclass[11pt]{article}
\usepackage{amsfonts,mathrsfs,amssymb,amsthm}
\usepackage{amsmath,amscd}
\usepackage{mathptm,pslatex}
\usepackage{url}
\usepackage{cite}
\usepackage{exscale}
\usepackage{relsize}
\usepackage[all]{xy}
\usepackage{color}

\begin{document}

\newtheorem{Def}{Definition}[section]
\newtheorem{Bsp}[Def]{Example}
\newtheorem{Prop}[Def]{Proposition}
\newtheorem{Theo}[Def]{Theorem}
\newtheorem{Lem}[Def]{Lemma}
\newtheorem{Koro}[Def]{Corollary}
\theoremstyle{definition}
\newtheorem{Rem}[Def]{Remark}

\newcommand{\rad}{{\rm rad}}
\newcommand{\pd}{{\rm projdim}}
\newcommand{\Db}[1]{{\mathscr D}^b(#1)}
\newcommand{\Kb}[1]{{\mathscr K}^b(#1)}
\newcommand{\Dsg}[1]{{\mathscr D}_{\rm sg}(#1)}
\newcommand{\modcat}[1]{#1\mbox{{\rm -mod}}}
\newcommand{\stmodcat}[1]{#1\mbox{{\rm -{\underline{mod}}}}}
\newcommand{\pmodcat}[1]{#1\mbox{{\rm -proj}}}
\newcommand{\Hom}{{\rm Hom}}
\newcommand{\Ext}{{\rm Ext}}
\newcommand{\Ker}{{\rm Ker}}
\newcommand{\soc}{{\rm soc}}
\newcommand{\Coker}{\mathop{\rm Coker}}
\newcommand{\StHom}{{\rm \underline{Hom}}}
\newcommand{\ra}{\rightarrow}
\newcommand{\raf}[1]{\stackrel{#1}{\ra}}
\newcommand{\lra}{\longrightarrow}
\newcommand{\lraf}[1]{\stackrel{#1}{\lra}}

{\Large \bf
\begin{center}
Examples of nontrivial presilting objects in a singularity category
\end{center}
}

\medskip
\centerline{\textbf{Changchang Xi and Jinbi Zhang$^*$}}

\renewcommand{\thefootnote}{\alph{footnote}}
\setcounter{footnote}{-1}
\footnote{$^*$Corresponding author's Email: zhangjb@ahu.edu.cn}
\renewcommand{\thefootnote}{\alph{footnote}}
\setcounter{footnote}{-1}
\renewcommand{\thefootnote}{\alph{footnote}}
\setcounter{footnote}{-1}

\begin{abstract}
It was conjectured that no presilting objects exist in the singularity category of a finite-dimensional algebra over a field except the zero object. In this note, we give counterexamples to this conjecture for algebras over fields of positive characteristic.
\end{abstract}

\section{Introduction}

For an Artin algebra $A$, Auslander--Reiten formulated the following conjecture in \cite{AuslanderReiten1975}:

(ARC) A finitely generated $A$-module $M$ with $\Ext_A^i(M,M\oplus A)=0$ for all $i>0$ is projective.

This conjecture is not yet solved. When we pass to the derived category $\Db{A}$ of $A$, the condition in (ARC) means that $\Hom_{\Db{A}}(A\oplus M,(A\oplus M)[i])=0$ for all $i>0$. Thus $A\oplus M$ is a presilting object in $\Db{A}$. Here, an object $X$ of a triangulated category $\mathcal{T}$ is \emph{presilting} if
$\Hom_{\mathcal{T}}(X,X[d])=0$ for all $d>0$. Furthermore, when we consider the singularity category $\Dsg{A}$ of $A$, which is, by definition, the Verdier quotient of $\Db{A}$ by the bounded homotopy category of finitely generated projective $A$-modules (see \cite{buchweitz87,orlov04}), then finitely generated projective $A$-modules are $0$ in $\Dsg{A}$. Based on this background, the following so-called singular presilting conjecture was naturally proposed in \cite[Section~1]{ChenHuQinWang2023}.

\medskip
(SPC) For an Artin algebra $A$, $\Dsg{A}$ does not contain any nonzero presilting objects.

\medskip
Trivially, (SPC) implies (ARC). It was shown in \cite[Theorem~2.8 and Propositions~3.4--3.5]{ChenLiZhangZhao2024}
that (SPC) holds for singularly-minimal and ultimately-closed algebras.
Whether (SPC) is true seems not to be studied in general.

The purpose of the present note is to show that (SPC) is not true in general by constructing an explicit counterexample.

\smallskip
Let $k$ be a field and $K$ a field extension of $k$ such that $K$ contains a transcendental element $t$ over $k$.
Let $k[T]$ be the polynomial algebra over $k$ in one variable $T$ and $F(T)=\sum_{i=0}^nc_iT^i\in k[T]$ with  $n\ge 2$ and $c_n\ne 0$.

We define a $K$-algebra
$A_F$ depending upon $F$ as follows.

Let $A_F:=K\langle x_0,\ldots,x_n\rangle/I_F$ be the quotient of the free $K$-algebra $K\langle x_0,\ldots,x_n\rangle$ of rank $n+1$ in variables $x_0, x_1,\cdots, x_n$, modulo the ideal $I_F$ generated by
$$
 \begin{gathered}
 c_nx_ix_j-c_jx_ix_n
 \quad(0\le i<n,\ 1\le j<n),
 \qquad
 x_nx_j
 \quad(1\le j\le n),\\
 c_nx_0^2-c_0x_0x_n,
 \qquad
 c_nx_ix_0-c_0x_ix_n+x_{i-1}x_n
 \quad(1\le i\le n).
 \end{gathered}
$$
Then $A_F$ is a finite-dimensional local $K$-algebra  and $\rad^3(A_F)=0$. Moreover, $A_F$ is of dimension $2n+2$ with a $K$-basis
$$\{1, x_0, x_1, \cdots, x_n,x_0 x_n, x_1x_n, \cdots, x_{n-1}x_n\}.$$ In fact, since the relations are quadratic, we only need to prove that $\{x_0 x_n, x_1x_n, \cdots, x_{n-1}x_n\}$ is $K$-linearly independent. Suppose that there are $\lambda_0,\cdots, \lambda_{n-1}\in K$ such that $\sum_{i=0}^{n-1}\lambda_ix_ix_n=0$ in $A_F$, that is, $\sum_{i=0}^{n-1}\lambda_ix_ix_n\in I_F$ in $K\langle x_0,\ldots,x_n\rangle$. We show $\lambda_i=0$ for all $0\le i< n$.
Note that $K\langle x_0,\ldots,x_n\rangle$ is a graded $K$-algebra with homogeneous $i$-th component being spanned by all monomials of degree $i$ for $i\ge 0$. Every
element of $I_F$ is a finite sum of elements of the form $urv$, where
$u,v\in K\langle x_0,\ldots,x_n\rangle$ and $r$ is one of the defining
relations. If either $u$ or $v$
is homogeneous of positive degree, then $urv$ has degree at least
three. By decomposing $u$ and $v$ into their homogeneous components, one knows that the degree-two part of $urv$ comes only from their constant terms.
It follows that every homogeneous element of degree two in $I_F$ is a
$K$-linear combination of the defining relations.
So we may write
$$
\begin{aligned}
 \sum_{i=0}^{n-1}\lambda_ix_ix_n
 ={}&\sum_{\substack{0\le i<n\\1\le j<n}}
 \alpha_{ij}(c_nx_ix_j-c_jx_ix_n)
 +\sum_{j=1}^n\beta_jx_nx_j\\
 &+\gamma(c_nx_0^2-c_0x_0x_n)
 +\sum_{i=1}^n\delta_i
 (c_nx_ix_0-c_0x_ix_n+x_{i-1}x_n),
\end{aligned}
$$
where $\alpha_{ij}, \beta_j, \delta_i$ and $\gamma$ lie in $K$.
By comparing the coefficients of $x_ix_j$, $x_nx_j$, $x_0^2$, and
$x_ix_0$ of both sides, we get $c_n\alpha_{ij}=0$, $\beta_j=0$,
$c_n\gamma=0$, and $c_n\delta_i=0$. Due to $c_n\ne0$, we see
$\alpha_{ij}=\beta_j=\delta_i=\gamma=0$ for all $i,j$. Therefore $\lambda_i=0$ for all $0\le i<n$. Thus $\{x_0 x_n, x_1x_n, \cdots, x_{n-1}x_n\}$ is a $K$-linearly independent system.

Since $A_F$ is a local algebra, it is easy to see that a finitely generated $A_F$-module is of finite projective dimension if and only if it is a free $A_F$-module of finite rank.

For $\lambda\in K$, we define
$h_\lambda:=\sum_{i=0}^n\lambda^ix_i\in A_F$,
$N_{F,\lambda}:=A_F/(A_Fh_\lambda)$, and
$M_F:=N_{F,t}=A_F/(A_Fh_t)$.

\begin{Theo}\label{thm-main}
The module $M_F$ is a nonzero presilting object in $\Dsg{A_F}$ if and only if $F'(T)=0$, where $F'(T)$ denotes the derivative of $F(T)$.

\end{Theo}

In particular, for every prime $p>0$, let $k=\mathbb F_p:=\mathbb{Z}/p\mathbb{Z}$,
$F(T)=T^p\in k[T]$, and let $K:=k(t)$ be the field of fractions of the polynomial algebra $\mathbb F_p[t]$ in one variable $t$. Then
$F'(T)=0$, and hence (SPC) fails for $A_F$. However,
(ARC) holds for $A_F$ by \cite[Theorem 1.1]{Xu2013}.

\section{Proof of Theorem \ref{thm-main}}\label{sect2}
Before we start the proof of  Theorem \ref{thm-main}, we introduce a few notions and notation.

Let $A$ be a finite-dimensional algebra over a field $k$. We denote by
$\modcat{A}$ the category of finitely generated left
$A$-modules, and by $\pmodcat{A}$ the full subcategory of $A$-mod consisting of all projective
modules.

The stable category of $A$ is denoted by $\stmodcat{A}$. It
has the same objects as $\modcat{A}$, and its morphism space
$\StHom_A(X,Y)$ is the quotient of $\Hom_A(X,Y)$ by the subspace of
homomorphisms which factorize through an object in $\pmodcat{A}$. For
$f\in\Hom_A(X,Y)$, we write $\underline f$ for its image in
$\StHom_A(X,Y)$. We denote by $\Omega_A$ the syzygy functor on
$\stmodcat{A}$.

The composition of two morphisms in a category is written from left to right.

In the rest of this note, we fix a polynomial $F(T)\in k[T]$ and write $N_{\lambda}$ for $N_{F,\lambda}$.  In the following, we give some properties of $A_F$ and $M_F$.

For $0\le i<n$, set
$w_i:=c_n^{-1}x_ix_n$, $E:=\bigoplus_{i=0}^nKx_i$ and
$W:=\bigoplus_{i=0}^{n-1}Kw_i$. Let $w_{-1}=w_n=0$. Then

\medskip
(1) $A_F=K1\oplus E\oplus W \mbox{  and } \rad(A_F)=E\oplus W $ as direct sums of $K$-spaces, $EW=WE=W^2=0, \rad^2(A_F)=W \mbox{ and } \rad^3(A_F)=0.$
The only nonzero products between the
displayed basis elements of $E$ are
$$
 x_ix_j=c_jw_i
 \quad(0\le i<n,\ 1\le j\le n),
 \qquad
 x_ix_0=c_0w_i-w_{i-1}
 \quad(0\le i\le n).
$$

\medskip
(2) $Eh_\lambda=W$, $\{y\in E\mid yh_\lambda=0\}=Kh_{F(\lambda)}$, $A_Fh_\lambda=Kh_\lambda\oplus W, \mbox{ and } Kh_{F(\lambda)}\oplus W
=A_Fh_{F(\lambda)}$.

\medskip
This follows from the following equalities $(\star)$ and the fact $(\star')$:
For
$y=\sum_{i=0}^ny_ix_i\in E$ with $y_i\in K$, we have
\begin{equation*}
 yh_\lambda=(\sum_{i=0}^ny_ix_i)(\sum_{i=0}^n\lambda^ix_i)
 =\sum_{i=0}^{n-1}\bigl(F(\lambda)y_i-y_{i+1}\bigr)w_i.
 \tag{$\star$}
\end{equation*}

$(\star')\;$ If
$a=\alpha+y+w\in K1\oplus E\oplus W$, then $a h_\lambda=\alpha h_\lambda+yh_\lambda.$

\medskip
Since we have $E\cap W=0$, the equality $ah_\lambda=0$ if and only if $\alpha=0$ and
$y\in Kh_{F(\lambda)}$. Consequently, we have (3):

\medskip
(3) $\{a\in A_F\mid a h_\lambda=0\}=Kh_{F(\lambda)}\oplus W
=A_Fh_{F(\lambda)}. $

\medskip
(4) Let $t_0:=t$ and $t_{m+1}:=F(t_m)$ for $m\ge0$. Then $\Omega_{A_F}^m(M_F)\simeq N_{t_m}$ for $m\ge0$.

\medskip
Indeed, let $\pi_\lambda:A_F\ra N_\lambda$ be the
canonical surjective homomorphism, and let $\iota_\lambda:N_{F(\lambda)}\ra A_F$ be the right multiplication map by $h_\lambda$, which sends $a+A_Fh_{F(\lambda)}$ to $ah_{\lambda}$ for $a\in A_F$. Then $\iota_\lambda$ is a homomorphism of $A_F$-modules. By (3),
$\iota_\lambda$ is a
well-defined injective homomorphism with $\operatorname{Im}(\iota_\lambda)=A_Fh_\lambda=\Ker(\pi_\lambda)$.
Hence there is an exact sequence of $A_F$-modules
$$
 0\lra N_{F(\lambda)}\lraf{\iota_\lambda}A_F
 \lraf{\pi_\lambda}N_\lambda\lra0.
$$
Due to the inclusion $A_Fh_\lambda\subseteq\rad(A_F)$, we see that
$\pi_\lambda$ is a projective cover of $N_\lambda$. Hence
$\Omega_{A_F}(N_\lambda)\simeq N_{F(\lambda)}$ as $A_F$-modules, and therefore $\Omega_{A_F}^m(N_{\lambda})\simeq N_{F^{[m]}(\lambda)}$, where $F^{[m]}(\lambda)$ stands for the $m$-fold iteration of $F(\lambda)$, that is, $F^{[m]}(\lambda)=\underbrace{F(\cdots F(F(\lambda))\cdots)}_m$ for $m\ge 1$. We understand $F^{[0]}(\lambda)=\lambda$. In particular,  $\Omega_{A_F}^m(M_F)\simeq N_{t_m}$ for $m\ge0$.

\medskip
(5) We compute the stable morphism spaces $\StHom_{A_F}(N_\lambda,N_\mu)$  for $\lambda, \mu\in K$.

\medskip
We define a  homomorphism of $K$-spaces:
$$\ell_F:E\lra K, \quad r=\sum_{i=0}^nr_ix_i \mapsto \sum_{i=0}^nr_ic_i, \mbox{ where all } r_i \mbox{ lie in } K.$$
For $a\in K$, we
define another homomorphism of $K$-spaces:
$$
 \mu_a:E\lra K,
 \quad
r=\sum_{i=0}^nr_ix_i \mapsto (r)\ell_F-ar_0, \mbox{ where all } r_i \mbox{ lie in } K.
$$
Then
\begin{equation*}
 h_ar=(r)\mu_a\sum_{i=0}^{n-1}a^iw_i \, \, \mbox{ for } a\in K, r\in E.
 \tag{$\star\star$}
\end{equation*}
Since the coefficient of $w_0$ in $\sum_{i=0}^{n-1}a^iw_i$ is $1$,
the formula $(\star\star)$ shows that $h_ar=0$ if and only if
$(r)\mu_a=0$.

For $a,b\in K$,
we have $(A_Fh_a)r\subseteq W\subseteq A_Fh_b$ for $r\in E$ by (1) and (2).
Hence the right multiplication by $r$ gives rise to a homomorphism
$f_r:N_a\lra N_b, \, \overline{c} \mapsto \overline{cr}$, of $A_F$-modules, where $\bar{c}:=c + A_Fh_a$ is the coset of $c\in A_F$.

\begin{Lem}\label{stable-hom-lemma}
Let $a,b\in K$ with $a\ne b$. Then
 there is an isomorphism of $K$-spaces
$$
 \Phi_{a,b}:E/(Kh_b+\Ker(\mu_a))\lra
 \StHom_{A_F}(N_a,N_b),\; \,
 r+(Kh_b+\Ker(\mu_a))\mapsto\underline{f_r}.
$$
In particular, there are isomorphisms of $K$-spaces:
$$
 \StHom_{A_F}(N_a,N_b)\simeq
 \Coker\bigl(Kh_b\lraf{\mu_a}K\bigr)
 \simeq \begin{cases}
 K,& a=F(b),\\
 0,&\mbox{otherwise}.
 \end{cases}
$$
\end{Lem}

{\it Proof.}
To construct $\Phi_{a,b}$, we consider the assignment
$$
 \begin{aligned}
 \Psi_{a,b}:E/Kh_b \lra\Hom_{A_F}(N_a,N_b),\;\;
 r+Kh_b \mapsto f_r.
 \end{aligned}
$$
This is well defined. Indeed, let $r,r'\in E$ with
$r-r'=\lambda h_b$ for some $\lambda\in K$. For
$\overline c\in N_a$ with $c\in A_F$, we have $c\lambda h_b\in A_Fh_b$, and hence
$(\overline c)f_r=\overline{cr}
=\overline{cr'+c\lambda h_b}
=\overline{cr'}=(\overline c)f_{r'}$ in $N_b$. Thus $f_r=f_{r'}$, so $\Psi_{a,b}$ is a well-defined $K$-linear
homomorphism. We show that $\Psi_{a,b}$ is an isomorphism of $K$-spaces.

To prove the surjectivity, we pick up a homomorphism $f:N_a\ra N_b$ of $A_F$-modules. Thanks to
$A_F=K1\oplus E\oplus W$ and $W\subseteq A_Fh_b$, we can write
$(\overline1_a)f=\alpha\overline1_b+\overline s$, with
$\alpha\in K$ and $s\in E$. Thus it follows from $h_a\overline1_a=0$ and $h_as\in W\subseteq A_Fh_b$ that
$
0=(h_a\overline1_a)f
=\alpha\overline{h_a}+\overline{h_as}
=\alpha\overline{h_a}.
$
By (2), $(A_Fh_b)\cap E=Kh_b$. By the assumption $a\ne b$, we have
$h_a\notin Kh_b$, and hence $\overline{h_a}\ne0$ in $N_b$.
Thus $\alpha=0$ and $f=f_s$. This shows that $\Psi_{a,b}$ is surjective.

To prove the injectivity, we suppose $f_r=0$. Then
$\overline r=(\overline1_a)f_r=0$ in $N_b$, and hence
$r\in(A_Fh_b)\cap E=Kh_b$. Thus $\Psi_{a,b}$ is injective
and hence is an isomorphism of $K$-spaces.

We now determine which homomorphisms from $N_a$ to $N_b$ factorize
through projective $A_F$-modules. Since $A_F$ is local, every finitely generated projective
$A_F$-module is free, and hence is isomorphic to $A_F^q$ for some
nonnegative integer $q$. If a homomorphism $f:N_a\to N_b$ factorizes through $A_F^q$,
then $f$ can be written as $f=\sum_{j=1}^q g_jv_j$, where $g_j: N_a\to A_F$ and $v_j: A_F\to N_b$ are homomorphisms of $A_F$-modules for all $j$. Thus it is enough to consider
$f: N_a\to N_b$ of the form $f=gv $:
$$ N_a\lraf{g}A_F\lraf{v}N_b.
$$
Write $(\overline1_a)g=\alpha+e+w$ with $\alpha\in K$, $e\in E$,
and $w\in W$. It follows from $h_a\overline1_a=0$ that
$$
 0=(h_a\overline1_a)g
 =h_a(\alpha+e+w)
 =\alpha h_a+h_ae.
$$
Here, we have used the fact $EW=0$. Then it yields from $\alpha h_a\in E$, $h_ae\in W$,
and $E\cap W=0$ that $\alpha=0$ and $h_ae=0$. Since the
coefficient of $w_0$ in $\sum_{i=0}^{n-1}a^iw_i$ is $1$, the formula
$(\star\star)$ shows that $h_ae=0$ if and only if $(e)\mu_a=0$ if and only if
$e\in\Ker(\mu_a)$. Now we write $(1)v=\beta\overline1_b+\overline u$ with $\beta\in K$ and
$u\in E$. It follows from $E^2\subseteq W\subseteq A_Fh_b$ and
$W\rad(A_F)=0$ that
$
 (\overline1_a)(gv)
 =(e+w)(\beta\overline1_b+\overline u)
 =\beta\overline e.
$
Thus $gv=f_{\beta e}$. Due to $\beta e\in\Ker(\mu_a)$, every
homomorphism which factorizes through a projective $A_F$-module is
represented by an element of $\Ker(\mu_a)$.

Conversely, let $e\in\Ker(\mu_a)$. The formula $(\star\star)$ gives
$h_ae=0$, and hence $(A_Fh_a)e=0$. Thus
$(\overline c)g_e:=ce$ defines a homomorphism
$g_e:N_a\ra A_F$ of $A_F$-modules such that $g_e\pi_b=f_e$. Therefore $f_e$ factorizes through the projective module $A_F$ for
$e\in\Ker(\mu_a)$.

Under the isomorphism $\Psi_{a,b}$, the homomorphisms which factorize
through projective modules correspond precisely to
$
 \bigl(Kh_b+\Ker(\mu_a)\bigr)/Kh_b.
$
Consequently, the following isomorphism of $K$-spaces holds:
$$
 \begin{aligned}
 \StHom_{A_F}(N_a,N_b)
 \simeq
 \frac{E/Kh_b}
 {\bigl(Kh_b+\Ker(\mu_a)\bigr)/Kh_b}\simeq E/\bigl(Kh_b+\Ker(\mu_a)\bigr).
 \end{aligned}
$$
This gives a way to define $\Phi_{a,b}$.

It remains to compute $E/(Kh_b+\Ker(\mu_a))$. Thanks to $(x_n)\mu_a=c_n\ne0$, the $K$-linear map
$\mu_a:E\ra K$ is surjective.
Moreover,
$
 (h_b)\mu_a=(h_b)\ell_F-a=F(b)-a.
$
Due to $(Kh_b)\mu_a=(F(b)-a)K$, the composition map
$$
 E\lraf{\mu_a}K\lra K/(F(b)-a)K
$$
is surjective and has kernel $Kh_b+\Ker(\mu_a)$. Hence there is an isomorphism of $K$-spaces:
$$
 \begin{aligned}
 E/(Kh_b+\Ker(\mu_a))
 &\lra K/(F(b)-a)K,\\
 r+(Kh_b+\Ker(\mu_a))
 &\mapsto (r)\mu_a+(F(b)-a)K.
 \end{aligned}
$$
Moreover,
$\Coker(Kh_b\lraf{\mu_a}K)
=K/(Kh_b)\mu_a=K/(F(b)-a)K$.
Since $K$ is a field, we have $K/(F(b)-a)K\simeq K$ if $a=F(b)$, and
$K/(F(b)-a)K=0$ otherwise.
$\square$

\medskip
(6) We compute homomorphisms between
consecutive modules.

\medskip
Let $X$ and $Y$ be $A_F$-modules, and $d>0$ a fixed integer. Since
$\Omega_{A_F}$ is an endofunctor on $\stmodcat{A_F}$,
it induces the following homomorphism of $K$-spaces for each $m\ge d$:
$$
 \theta_m:\StHom_{A_F}(\Omega_{A_F}^m(X),\Omega_{A_F}^{m-d}(Y))
 \lra
 \StHom_{A_F}(\Omega_{A_F}^{m+1}(X),\Omega_{A_F}^{m-d+1}(Y)), \; \,
 \underline f \longmapsto\underline{\Omega_{A_F}(f)}.
$$
More explicitly,  the homomorphism $\theta_m$ can be described as follows. Let
$f:\Omega_{A_F}^m(X)\ra\Omega_{A_F}^{m-d}(Y)$
be an $A_F$-homomorphism. Choose projective covers
$\pi_X:P\ra\Omega_{A_F}^m(X)$ and $\pi_Y:Q\ra\Omega_{A_F}^{m-d}(Y)$ of $\Omega_{A_F}^m(X)$ and $\Omega_{A_F}^{m-d}(Y)$,
respectively,
and let $\widetilde f:P\ra Q$ be a lift of $f$
such that $\widetilde f\pi_Y=\pi_Xf$.  Due to $\iota_X\pi_X=0$, we have
$
 \iota_X\widetilde f\pi_Y
 =\iota_X\pi_Xf=0.
$
Since $\iota_Y$ is the kernel of $\pi_Y$, there exists a unique
homomorphism
$
 \Omega_{A_F}(f):\Omega_{A_F}^{m+1}(X)
 \ra\Omega_{A_F}^{m-d+1}(Y)
$
of $A_F$-modules such that
$
 \Omega_{A_F}(f)\iota_Y=\iota_X\widetilde f.
$
Thus we have the commutative diagram:
\begin{equation*}
\vcenter{\hbox{\(
\xymatrix{
0\ar[r]&\Omega_{A_F}^{m+1}(X)\ar[r]^-{\iota_X}
 \ar[d]^-{\Omega_{A_F}(f)}
 &P\ar[r]^-{\pi_X}\ar[d]^-{\widetilde f}
 &\Omega_{A_F}^m(X)\ar[r]\ar[d]^-{f}&0\\
0\ar[r]&\Omega_{A_F}^{m-d+1}(Y)\ar[r]^-{\iota_Y}
 &Q\ar[r]^-{\pi_Y}
 &\Omega_{A_F}^{m-d}(Y)\ar[r]&0.
}
\)}}
\tag{$\ddagger$}
\end{equation*}
Although
$\Omega_{A_F}(f)$ depends on the choice of $\widetilde f$ in $\modcat{A_F}$, it is independent in $\stmodcat{A_F}$.
 Thus the image of $\underline{f}$ under $\theta_m$ is defined to be
$\underline{\Omega_{A_F}(f)}$ in $(\ddagger)$.

Now, let $X=Y=M_F$ and $d=1$. We can identify
$\Omega_{A_F}^j(M_F)$ with $N_{t_j}$ for $j\ge0$ by (4). Then the homomorphism
$\theta_m$ takes the following form for $m\ge1$:
$$
 \theta_m:\StHom_{A_F}(N_{t_m},N_{t_{m-1}})
 \lra
 \StHom_{A_F}(N_{t_{m+1}},N_{t_m}),
 \;  \underline f\mapsto\underline{\Omega_{A_F}(f)}.
$$

\begin{Lem}\label{stable-system-lemma}
For $m\ge d\ge1$, we have $$ \dim_K\StHom_{A_F}(N_{t_m},N_{t_{m-d}})=
 \begin{cases}1,& d=1,\\ 0,&d\ge2.\end{cases}$$
Particularly, for $m\ge d=1$, let
$$ \gamma_m:=-\frac{(t_m-c_0)F'(t_m)}{F(t_m)-c_0}\in K,
$$
where $F'(T)$ denotes the derivative of $F(T)\in K[T]$. Then there is a $K$-basis element $\underline{f_m}$ of $\StHom_{A_F}(N_{t_m},N_{t_{m-1}})$ such that $(c\underline{f_{m}})\theta_m
=c\gamma_m\underline{f_{m+1}}$ for all $c\in K$.
\end{Lem}

{\it Proof.} Note that $\gamma_m$ is well defined since the transcendentality of $t\in K$ over $k$ implies that $F(t_m)-c_0\ne 0.$

Fix $m\ge d\ge1$. As a polynomial in $t$, the
element $t_i$ has degree $n^i$ for integer $i\ge0$.
Hence $t_m\ne t_{m-d}$. If $d\ge2$, then it follows from $F(t_{m-1})=t_m$
that
$F(t_{m-d})=t_{m-d+1}\ne t_m$. Thus
Lemma~\ref{stable-hom-lemma} shows
$$
 \dim_K\StHom_{A_F}(N_{t_m},N_{t_{m-d}})
 =
 \begin{cases}
 1,&d=1,\\
 0,&d\ge2.
 \end{cases}
$$

We now consider the case $d=1$. For $m\ge1$, let
$f_m:N_{t_m}\ra N_{t_{m-1}}$ be the homomorphism $f_{x_0}$.
We show $\underline{f_{x_0}}\ne 0$ in
$\StHom_{A_F}(N_{t_m},N_{t_{m-1}})$. Indeed, under the isomorphism
$\Phi_{t_m,t_{m-1}}$ in Lemma~\ref{stable-hom-lemma}, the image of
$x_0+ (Kh_{t_{m-1}}+\Ker(\mu_{t_m}))$ under $\Phi_{t_m,t_{m-1}}$ is $\underline{f_{x_0}}$. Thus it suffices to prove $x_0\notin Kh_{t_{m-1}}+\Ker(\mu_{t_m})$. However, by the definition of
$\mu_{t_m}$ in (5), we have
$(h_{t_{m-1}})\mu_{t_m}=F(t_{m-1})-t_m=0$, whereas
$(x_0)\mu_{t_m}=c_0-t_m\ne0$. Hence
$Kh_{t_{m-1}}\subseteq\Ker(\mu_{t_m})$ and
$x_0\notin Kh_{t_{m-1}}+\Ker(\mu_{t_m})=\Ker(\mu_{t_m})$. Thus $\underline{f_{x_0}}\ne0$.

Similarly, as done in $(\ddagger)$, we get $(\underline{f_m})\theta_m
 =\underline{\Omega_{A_F}(f_m)}$, which is induced from
the following commutative diagram of $A_F$-modules:
$$\xymatrix{
0\ar[r]&N_{t_{m+1}}\ar[r]^-{\iota_{t_m}}
 \ar[d]^-{\Omega_{A_F}(f_m)}
 &A_F\ar[r]^-{\pi_{t_m}}\ar[d]^-{\rho_{x_0}}
 &N_{t_m}\ar[r]\ar[d]^-{f_m}&0\\
0\ar[r]&N_{t_m}\ar[r]^-{\iota_{t_{m-1}}}
 &A_F\ar[r]^-{\pi_{t_{m-1}}}
 &N_{t_{m-1}}\ar[r]&0,
}$$
where $\rho_{x_0}:A_F\ra A_F$ is the right multiplication map by $x_0$.

\medskip
We next calculate $\Omega_{A_F}(f_m)$. Let
$\bar{a}=a+A_Fh_{t_{m+1}} \in N_{t_{m+1}}=A_F/A_Fh_{t_{m+1}}$ with $a\in A_F$. The commutativity of the left
square in the above diagram shows
$ (\bar{a})
 \bigl(\Omega_{A_F}(f_m)\iota_{t_{m-1}}\bigr)
 =(\bar{a})(\iota_{t_m}\rho_{x_0})
 =a h_{t_m}x_0.$
By $(\star\star)$ and the definition of $\mu_{t_m}$ in
(5), we have
$h_{t_m}x_0
 =(x_0)\mu_{t_m}\big(\sum_{i=0}^{n-1}t_m^iw_i\big)
 =-(t_m-c_0)\big(\sum_{i=0}^{n-1}t_m^iw_i\big).$
Set $s:=(t_m-c_0)\big(\sum_{i=1}^n it_m^{i-1}x_i\big)\in E.$ It follows from $F(t_{m-1})=t_m$
 and $(\star)$ (with $\lambda=t_{m-1}$) that
$$
 s h_{t_{m-1}}
 =(t_m-c_0)\left(-w_0+
 \sum_{i=1}^{n-1}(i-(i+1))t_m^iw_i\right)
 =-(t_m-c_0)\big(\sum_{i=0}^{n-1}t_m^iw_i\big).
$$
Consequently, $ (\bar{a})
 \bigl(\Omega_{A_F}(f_m)\iota_{t_{m-1}}\bigr)
 =a h_{t_m}x_0
 =a s h_{t_{m-1}}
 =(\bar{a})(f_s\iota_{t_{m-1}}).
$
Since $\iota_{t_{m-1}}$ is injective, it follows that
$\Omega_{A_F}(f_m)=f_s$. Therefore $(\underline{f_m})\theta_m=\underline{f_s}.$

Since $\underline{f_{m+1}}$ is a $K$-basis element of
$\StHom_{A_F}(N_{t_{m+1}},N_{t_m})$, there is a unique
$\alpha_m\in K$ such that
$
 \underline{f_s}=\alpha_m\underline{f_{m+1}}.
$
By Lemma~\ref{stable-hom-lemma}, there is an isomorphism of $K$-spaces
$$
 \Phi_{t_{m+1},t_m}:
 E/(Kh_{t_m}+\Ker(\mu_{t_{m+1}}))
 \lra\StHom_{A_F}(N_{t_{m+1}},N_{t_m}),\;\, \;
 r+(Kh_{t_m}+\Ker(\mu_{t_{m+1}}))
 \mapsto\underline{f_r}.
$$
Under this isomorphism, $x_0$ corresponds to
$\underline{f_{m+1}}$. Hence it follows from
$\underline{f_s}=\alpha_m\underline{f_{m+1}}$ that
$
 s-\alpha_mx_0\in
 Kh_{t_m}+\Ker(\mu_{t_{m+1}}).
$
Since the equality $(h_{t_m})\mu_{t_{m+1}}=F(t_m)-t_{m+1}=0$ holds, we have
$Kh_{t_m}\subseteq\Ker(\mu_{t_{m+1}})$. Thus
$s-\alpha_mx_0\in\Ker(\mu_{t_{m+1}})$ and
$ (s)\mu_{t_{m+1}}
 =\alpha_m(x_0)\mu_{t_{m+1}}.
$
By the definitions of $\ell_F$ and $\mu_{t_{m+1}}$ in (5), we have
$(s)\mu_{t_{m+1}}=(t_m-c_0)F'(t_m)$ and
$(x_0)\mu_{t_{m+1}}=c_0-F(t_m)$. Therefore
$$
 \alpha_m=\frac{(s)\mu_{t_{m+1}}}{(x_0)\mu_{t_{m+1}}}
 =-\frac{(t_m-c_0)F'(t_m)}{F(t_m)-c_0}
 =\gamma_m.
$$
Thus
$(\underline{f_m})\theta_m
=\gamma_m\underline{f_{m+1}}$. Since $\theta_m$ is $K$-linear, we have
$(c\underline{f_m})\theta_m
=c\gamma_m\underline{f_{m+1}}$
for all $c\in K$.
$\square$

\medskip
Now we are ready to prove the main result.

\medskip
{\bf Proof of Theorem \ref{thm-main}.}
Let $d>0$ be a fixed integer, and let $X$ and $Y$ be $A_F$-modules. We have defined the following $K$-linear map for $m\ge d$:
$$
 \theta_m:\StHom_{A_F}(\Omega_{A_F}^m(X),\Omega_{A_F}^{m-d}(Y))
 \lra
 \StHom_{A_F}(\Omega_{A_F}^{m+1}(X),\Omega_{A_F}^{m-d+1}(Y)), \;
 \underline f \mapsto\underline{\Omega_{A_F}(f)}.
$$

For $r\ge m$, by the successive applications of $\theta_m$, we have a
homomorphism
$$
\begin{aligned}
 \theta_{m,r}:\StHom_{A_F}(\Omega_{A_F}^m(X),\Omega_{A_F}^{m-d}(Y))
 &\lra
 \StHom_{A_F}(\Omega_{A_F}^r(X),\Omega_{A_F}^{r-d}(Y)),\;
 \underline f \mapsto
 \underline{\Omega_{A_F}^{r-m}(f)}.
\end{aligned}
$$
The functoriality of $\Omega_{A_F}$ implies that $\theta_{m,m}=1$ and
$\theta_{m,r}\theta_{r,s}=\theta_{m,s}$ for $d\le m\le r\le s$.
Therefore, for each fixed integer $d>0$, the datum
$(\{\StHom_{A_F}
(\Omega_{A_F}^m(X),\Omega_{A_F}^{m-d}(Y))\}_{m\ge d},
\{\theta_{m,r}\}_{d\le m\le r})$ is a direct system
of $K$-spaces.

By
\cite[p. 809]{Kvamme2021}, there is a natural $K$-linear isomorphism of $K$-spaces:
$$
 \Hom_{\Dsg{A_F}}(X,Y[d])\simeq
 \varinjlim_{m\ge d}\StHom_{A_F}(\Omega_{A_F}^m(X),\Omega_{A_F}^{m-d}(Y))
 \qquad(d>0).
$$
We take $X=Y=M_F$. Since
$\Omega_{A_F}^m(M_F)\simeq N_{t_m}$ for every $m\ge0$,
the above isomorphism gives
\begin{equation*}
 \Hom_{\Dsg{A_F}}(M_F,M_F[d])\simeq
 \varinjlim_{m\ge d}
 \StHom_{A_F}(N_{t_m},N_{t_{m-d}}). \tag{$\ast$}
\end{equation*}
By Lemma \ref{stable-system-lemma}, $\StHom_{A_F}(N_{t_m},N_{t_{m-d}})=0$ if $d\ge2$, and  $\dim_K\StHom_{A_F}(N_{t_m},N_{t_{m-1}})=1$ if $d=1$.
Suppose $d=1$. If $F'=0$, then $\gamma_m=0$ for $m\ge1$.
It follows that the direct system in $(\ast)$ is isomorphic to
$$
 K\lraf{0}K\lraf{0}K\lraf{0}\cdots,
$$
whose direct limit is zero.
Since $t_m$ is a nonconstant
polynomial in the transcendental element $t$, it is
transcendental over $k$, and therefore $t_m-c_0\ne0$. Suppose $F'\ne0$. Then $F'(t_m)\ne0$ and
$\gamma_m\ne0$ for $m\ge 1$. It follows that
$\theta_m$ is an isomorphism for $m\ge 1$, and the direct system in $(\ast)$
is isomorphic to
$$
 K\lraf{\sim}K\lraf{\sim}K\lraf{\sim}\cdots,
$$
whose direct limit is isomorphic to $K$.
Hence $M_F$ is presilting if and
only if $F'=0$. As $M_F$ is not a projective $A_F$-module, it has infinite projective dimension, and therefore $0\ne M_F$ in $\Dsg{A_F}$. Thus $M_F$ is a nonzero presilting object
in $\Dsg{A_F}$ if and only if $F'=0$.
$\square$

\section*{Acknowledgements}

The research was partially supported by the National Natural Science
Foundation of China (Grants 12671048 and 12401038). For calculations of the counterexample, we asked for assistance from OpenAI GPT-5.6 Sol.

{\footnotesize

Changchang Xi, School of Mathematical Sciences, Capital Normal
University, 100048 Beijing, P. R. China

{\tt Email: xicc@cnu.edu.cn (C. C. Xi)}

\medskip
Jinbi Zhang, School of Mathematical Sciences, Anhui University,
230601 Hefei, P. R. China

{\tt Email: zhangjb@ahu.edu.cn (J. B. Zhang)}
}


\begin{thebibliography}{99}
\bibitem{AuslanderReiten1975}{
{\sc M. Auslander} and {\sc I. Reiten}, On a generalized version of the Nakayama conjecture, \emph{Proc. Amer. Math.
Soc.} \textbf{52} (1975) 69-74.}

\bibitem{buchweitz87}{{\sc R.-O. Buchweitz}, Maximal Cohen-Macaulay modules and Tate cohomology over Gorenstein rings, Mathematical Surveys and Monographs, vol. 262, American Mathematical Society, Providence, RI, 2021.}


\bibitem{ChenLiZhangZhao2024}{
{\sc X.-W. Chen}, {\sc Z.-W. Li}, {\sc X. J. Zhang} and {\sc Z. B. Zhao},
A non-vanishing result on the singularity category,
\emph{Proc. Amer. Math. Soc.} \textbf{152} (2024) 3765-3776.}

\bibitem{ChenHuQinWang2023}{
{\sc Y. P. Chen}, {\sc W. Hu}, {\sc Y. Y. Qin} and {\sc R. Wang},
Singular equivalences and Auslander--Reiten conjecture,
\emph{J. Algebra} \textbf{623} (2023) 42-63.}


\bibitem{Kvamme2021}{
{\sc S. Kvamme},
$d\mathbb Z$-cluster tilting subcategories of singularity categories,
\emph{Math. Z.} \textbf{297} (2021) 803-825.}

\bibitem{orlov04}{{\sc D. O. Orlov}, Triangulated categories of singularities and D-branes in Landau-Ginzburg models, \emph{Proc. Steklov Inst. Math.} \textbf{246}(3) (2004) 227-248.}

\bibitem{Xu2013}{
{\sc D. M. Xu}, A note on the Auslander--Reiten conjecture,
\emph{Acta Math. Sin. (Engl. Ser.)} \textbf{29} (2013) 1993-1996.}


\end{thebibliography}
\end{document}